\documentclass[11pt,twoside,a4paper]{article}
\usepackage{amsmath}
\usepackage{amssymb, amsfonts, amsmath, amsthm, amscd}
\title{Normal generation of
 line bundles \\on multiple coverings}
\author{Seonja Kim\thanks{This work was supported by the
Korea Research Foundation Grant funded by the Korean Government
(2005-070-C00005).}\\
Department of Digital Broadcasting $\&$ Electronics\\ Chungwoon University \\
Chungnam, 350-701, Korea \\
e-mail: sjkim{\char "40}chungwoon.ac.kr }
\date{}

\def\cli{\mbox{Cliff}}
\def\deg{\mbox{deg}}

\def\ci{\mathcal I}
\def\co{\mathcal O}
\def\cl{\mathcal L}
\def\ck{\mathcal K}
\def\cm{\mathcal M}
\def\cn{\mathcal N}
\def\ca{\mathcal A}

\def\O{\mathcal O}
\def\blt{\blacktriangle}
\def\bltd{\blacklozenge}

\def\cli{\hbox{\rm Cliff}}

\begin{document}

\newtheorem{lem}{Lemma}[section]
\newtheorem{thm}[lem]{Theorem}
\newtheorem{claim}{Claim}
\newtheorem{prop}{Proposition}
\newtheorem{Prop}[lem]{Proposition}
\newtheorem{rmk}[lem]{Remark}
\newtheorem{cor}[lem]{Corollary}
\newtheorem{exm}[lem]{Example}
\renewcommand{\theprop}{\Alph{prop}}
\maketitle

\begin{abstract}
\noindent Any line bundle $\cl $  on a smooth curve $C$ of genus
$g$ with  $\deg \cl \ge 2g+1$ is normally generated, i.e.,
$\varphi _\cl (C)\subseteq \mathbb P H^0 (C,\cl)$ is projectively
normal. However, it has known that more various line bundles of
degree $d$ failing to be normally generated  appear on multiple
coverings of genus $g$ as $d$ becomes smaller than $2g+1$.
 Thus, investigating  the normal generation
of line bundles on multiple coverings can be an effective approach
to the normal generation. In this paper, we obtain conditions for
line bundles on multiple coverings being normally generated or
not, respectively.

\end{abstract}

\noindent{\small{\bf{Mathematics Subject Classifications (2001)}}:
14H45, 14H10, 14C20.}

\noindent{\small{\bf{Key words}}: algebraic curve, multiple
covering, line bundle,  linear series, projectively normal, normal
generation.}

\setcounter{section}{0}
\section{Introduction}

Throughout this paper,  $C$  is a smooth irreducible algebraic
curve of genus $g$ over an algebraically closed field of
characteristic 0. A line bundle  $\cl$ on  $C$ is said to be
normally generated if $\cl$ is very ample and $\varphi _\cl (C)$
is projectively normal  for its associated  morphism $\varphi _\cl
: C \rightarrow \mathbb P H^0 (C,\cl)$.

Any line bundle of degree at least $2g+1$ is normally generated
and  a general line bundle of degree $2g$ on a non-hyperelliptic
curve  is normally generated(\cite{Ca}, \cite{Mat}, \cite{Mum}).
And hyperelliptic curves have no normally generated line bundles
of degree less than than $2g+1$(\cite{LM1}). Thus a natural
interest is to characterize line bundles of degree near $2g$
failing to be normally generated and curves carrying such line
bundles. In \cite{GL}, \cite{KK2}, \cite{KKO}  and \cite{MS}, they
determined conditions for   a nonspecial very ample line bundle
$\cl$ with $\deg \cl\ge 2g-5$ and a special very ample line bundle
$\cl$ with $\deg \cl\ge 2g-7$ failing to be normally generated,
respectively.
 Through those  results, it is notable that line bundles
failing to be normally generated appear on multiple covering
curves and are closely connected with line bundles on the base
curves of the coverings. Moreover, both the degrees of those
coverings  and the genera of base  curves become larger as
proceeding with those works. Thus, investigating the normal
generation of line bundles on multiple coverings is  a natural
approach to the normal generation.

The purpose of this work is to detect   conditions for line
bundles on  multiple  coverings  being normally generated or
failing to be normally generated, respectively. For a nonspecial
line bundle $\cl$ on a multiple  covering $C$ of $C'$, we
introduce, in section 4, a kind
  of concrete description such as
  \begin{equation}\label{*}\tag{*}  \cl\sim \ck _C-\phi^{*}g^0_t -B+E
  \end{equation}
    for some $B \ge 0, E> 0$ on $C$  and $g^0_t$ on $C'$
    satisfying    $h^0 (C, \phi^{*} g^0_t +B)=1$, ${\rm supp}(\phi^{*} g^0 _t +B  ) \cap
     {\rm supp}(E)=\emptyset$ and $B\ngeq \phi^{*}Q$ for any  $Q \in C'$.
  Here, $\phi$ is its covering morphism. Our results on  nonspecial line bundles
  are described in terms  of $B$ and $E$. Using this description,
  we also construct nonspecial line bundles
  possessing an intended normal generation property on multiple coverings.

 The results of this work are as follows. Here, we assume that $C$  admits
  an  $n$-fold covering morphism $\phi :C \rightarrow C'$
 for a smooth curve $C'$ of genus $p$.
  For the following (1) and (2), we additionally
 assume that  $\phi$ is simple with $g>np$ and define numbers
  $\mu:=[\frac {2n(n-1)p}{g-np}]$ and $\delta := \min \{ \frac{g}{6},
 ~ \frac{g-np}{n-1} -2, ~ \frac{2(2n+\mu -3)}{(2n+\mu
    -1)^2}g~\}$.

  \noindent{(1)}
  Let $\cl$ be a nonspecial
  line bundle on $C$ with $h^0 (C,\cl ) \ge 3$. We may set
  $\cl\sim \ck _C-\phi^{*}g^0_t -B+E$  as (*).
Then, $\cl$ is  normally generated if $B = \sum _{i=1}^{b}\phi^{*}
( Q_i)-P_i$ with $\phi(P_i )=Q_i$, $\deg E> b+2$ and
 $\deg \cl > 2g+1-\delta $. Specifically, any nonspecial line bundle
  $\cl\sim \ck _C-\phi^{*}g^0_t -B+E$ as
   (*) on  a double covering $C$
 with $3g >8(p+1)$ is  normally
  generated in case $\deg E> \deg B +2$ and  $\deg \cl >2g+1-\frac{g}{6}$.

 \noindent{(2)} Let  $\cl$ be   a special very ample line bundle on $C$
  with $\deg \cl >\frac{3g-3}{2}$.
 Assume $\ck _C \otimes \cl^{-1}=\phi ^{*} \cn (-\sum ^{b} _{i=1 }P_i)$
  for some line bundle $\cn$ on $C'$ and  $\sum ^{b} _{i=1 }P_i$
 on $C$ such that  $\sum ^{b} _{i=1 } \phi (P_i) \leq \cn$ and
 $P_i+P_j \nleqslant \phi^{*} (\phi
 (P_i))$.  Then  $\cl$ is normally generated
 if $b\le 3$ and    $\deg \cl > 2g +1-2h^1 (C,\cl) -\delta$.
 (Note that the condition $P_i+P_j \nleqslant \phi^{*} (\phi
 (P_i))$ is satisfied if the points of $\sum ^{b} _{i=1 }\phi (P_i )$
 are distinct.) In particular, any line bundle $\cl$
 on a double covering with $3g > 8(p+1)$ such that
 $\ck _C \otimes \cl ^{-1}  \sim  \phi^{*} \cm\otimes \co _C (B)$, $B \ngeq \phi^{*}Q$
 for any  $Q \in C'$, then $\cl $ is normally generated in case
 $\deg B \leq 3$ and $\deg  \cl >2g +1 -\frac{g}{6} -2h^1 (C, \cl
 )$.

For each of these results,  we also obtain
  examples of   line bundles failing to be normally generated  on   multiple
  coverings  in case the number $b$  is lying on the
  outside boundary of the condition, i.e.,  $\deg E =b+2$, $b=4$, respectively.
  Hence those conditions  are sharp in some sense.

Using the result (1),  for any
 $d> \mbox{max}\{ 2g-2p,~ 2g+1 -\frac{g}{6}\}$
 we construct nonspecial normally
generated line bundles  of degree $d$ on double coverings  with $g
\ge 4p$(: See Corollary \ref{thm:theorem11}, whose result also
contains the cases $n\ge 3$.).
 On the one hand, a double covering of genus g admits no nonspecial normally
generated line bundles $\cl$ with $g+5 \le \deg \cl \le
2g-3p$(\cite{LM1}).

\noindent{(3)} Let $\cl$ be a nonspecial  line bundle on $C$ with
$h^0 (C,\cl ) \ge 3$.  We may set $\cl\sim \ck _C-\phi^{*}g^0_t
-B+E$ as  (*). Assume $B = \sum _{i=1}^{b}\phi^{*} ( Q_i)-P_i$,
$\phi(P_i )=Q_i$ and $h^1 (C, \cl ^2 (-\sum ^{b} _{i=1 } P_i
-E))=0$.
  Then  $\cl$ is very ample and  fails to be  normally generated
   if $a\ge 3 $ and $a+b  >\frac{(a+b-r)(a+b-r-1)}{2}$,~ where
$a=\deg E$, $r=h^0 (C', g^0_t +\sum ^{b} _{i=1 }  Q_i )-1$.

\noindent{(4)} Let $\cl$ be a special  very ample  line bundle on
$C$.  Assume $\ck _C \otimes \cl ^{-1}=\phi ^{*} \cn (-\sum ^{b}
_{i=1 }P_i)$
  for some line bundle $\cn$ on $C'$ and  $\sum ^{b} _{i=1 }P_i$
 on $C$. Set $c:=h^0 (C,
\phi ^* \cn)  -h^0 (C, \phi ^{*} \cn (-\sum ^{b}_{i=1} P_i))$.
Then $\cl$ fails to be  normally generated if  $ b>
\frac{(b-c)(b-c+1)}{2}$ and   $h^1 (C, \cl ^2(- \sum ^{b} _{i=1
}P_i))=0$.

 \vskip5pt
 In this paper, $g^r_d$ means a linear series of dimension $r$
and degree $d$. In particular, $g^0_d$ also denotes the
corresponding effective divisor of degree $d$. The notation
$\cl-g^r_d$ means $\cl(-D)$ for $D\in g^r_d$ and a line bundle
$\cl$. For a divisor $D$ and a line bundle $\cl$ on a smooth curve
$C$, we also denote $h^i (C, {\cal O}_C (D))$ by $h^i (C,D)$ and
${\cal O} _C (D)\subseteq \cl$ by $D \le \cl$.  And $\ck _C$ means
the canonical line bundle on $C$. The Clifford index of a smooth
curve $C$ is defined by $\cli(C):= \min \{ \cli(\mathcal L) : h^0
(C, \mathcal L) \geq 2, \ h^1 (C, \mathcal L) \geq 2\}$, where
$\cli (\cl ) =\deg \cl -2h^0 (C, \cl ) +2$.

\section{Preliminaries}

Before going into main theorems, we consider some lemmas which
will be used in our study.

\begin{lem} Let $\cl$ be a very ample line bundle on
  a smooth curve $C$. Consider the embedding
$C\subset \mathbb P H^0(C,\cl )=\mathbb P^r$ defined by $\cl$.
Then $\cl$ fails to be normally generated if there exists an
effective divisor $D$ on $C$ such that ${\ deg} D >
\frac{(n+1)(n+2)}{2}$ and $H^1(C,\cl^2(-D))=0$, where $  n:=dim
\overline{D}$, $ \overline{D}:=\cap \{ H \in H^0 (\mathbb P^r,
\co_{\mathbb P^r} (1))~| ~ H.C \ge D \}$.\label{Prop:prop 1.6}
\end{lem}
\begin{proof}
 Set
$$\Psi:= \{ S \in H^0 (\mathbb P^r,
\co_{\mathbb P^r} (2)) ~| ~ S:
 \mbox {  quadric cone with vertex }\overline{D} \}.$$
 Then, $\Psi\subseteq H^0 (\mathbb P ^r, \ci _{D/\mathbb P ^r} (2))$~ and
  $$\mbox{dim}\Psi = \mbox{dim Grass}(r-n-1, r) +
h^0 ( \mathbb P ^{r-n-1} , \O _{\mathbb P ^{r-n-1}} (2))=
\frac{r^2 +3r -n^2 -3n}{2} .$$ This yields
$$h^0 (\mathbb P ^r , \O _{\mathbb P ^r} (2)) -
h^0 (\mathbb P ^r , \ci _{D/\mathbb P ^r} (2))\le  h^0 (\mathbb P
^r , \O _{\mathbb P ^r} (2)) -\mbox{dim}\Psi  <\deg D,$$ for $\deg
D>\frac{(n+1)(n+2)}{2}$. From this, we get $h^1 (\mathbb P ^r ,\ci
_{D/\mathbb P ^r} (2))\neq 0$ by the exact sequence $0 \rightarrow
\ci _{D/\mathbb P ^r} (2) \rightarrow \O _{\mathbb P ^r} (2)
\rightarrow \co _{D} (2) \rightarrow 0$. Considering the exact
sequence $0 \rightarrow \ci _{C/\mathbb P ^r} (2) \rightarrow \ci
_{D/\mathbb P ^r} (2) \rightarrow \ci _{D/C} (2) \rightarrow 0$
and $h^1 (\mathbb P ^r ,\ci _{D/C} (2))= h^1(C,\cl^2(-D))=0$, we
have $h^1 (\mathbb P ^r, \ci _{C/\mathbb P ^r} (2))\neq 0$, which
proves the result.
\end{proof}
 This is practical for verifying the  non-normal generation
  of line bundles on smooth curves, since its conditions
  are purely numerical
  and hence can be computed by  theories about linear series.
On the one hand,  we have the following lemma from the proof of
Theorem 3 in \cite {GL}, which
 is useful to determine the normal generation of   a line bundle, since
 it provides  another line bundle  with  higher speciality
   in case the line bundle fails to be normally generated.
\begin{lem}\label{rmk:remark2.2}
Let  $\cl$ be a very ample line bundle on $C$ with  $\deg
\cl>\frac{3g-3}{2}+\epsilon$, where $\epsilon=0$ if $\cl$ is
special, $\epsilon=2$ if $\cl$ is nonspecial. If $\cl$ fails to be
normally generated, then there exists a line bundle $\ca\simeq \cl
(-R), ~R>0$ such that \hbox{\rm (i)} $\cli (\ca )\le \cli (\cl )$,
\hbox{\rm (ii)} $\deg \ca\ge \frac{g-1}{2}$, \hbox{\rm (iii)}
$h^0(C,\ca)\ge 2$ and $h^1(C,\ca)\ge h^1(C,\cl)+2$.
\end{lem}

Since the lemma plays an important role in this work, we
frequently have to compute the Clifford indices of line bundles.
Thus,  preparing the following is effective to prove the main
results.

\begin{lem}
Let  $\cm$ be a base point free line bundle on $C$  with
deg$\cm\le 2g-2$ such that its associated morphism $\varphi
_{\cm}$ is birational.

$(i)$ If $\deg \cm\ge g-1$, then $\cli (\cm )\ge \dfrac{\deg ~\ck
_C \otimes \cm ^{-1}}{3}$.

$(ii)$ If $\deg \cm\le g-1$, then
\begin{eqnarray*}
\cli (\cm )&\ge &\frac{g}{3} -1\qquad~~~~~~~\mbox{ for } ~~l=2,\\
\cli (\cm )&>&\frac{2(l-1)}{(l+1)^2}g-1 ~~~\mbox{ for } ~~l\ge 3 ,
\end{eqnarray*}
where $l :=\left [\frac{2g}{\deg \cm-1}\right ]$. \label{Prop:prop
1.7}
\end{lem}

\begin{proof} Set $\alpha :=h^0 (C, \cm ) -1$ and $d:=\deg \cm$.
First, assume $d\ge g-1$. Then  $\alpha\le \frac{2d-g+1}{3}$ by
Castelnuovo's genus bound, and hence
$$\mbox{Cliff}(\cm)=d-2\alpha\ge
\frac{2g-d-2}{3}=\frac{\mbox{deg}~\ck _C \otimes \cm ^{-1}}{3}.$$

Next, suppose $d\le g-1$. Set $l=\left [\frac{2g}{d-1}\right ]$.
Then  $\frac{2g}{l}+1\ge d>\frac{2g}{l+1}+1$. If $l=2$, then
Castelnuovo's genus bound yields $\alpha\le \frac{3d-g+3}{6}$(:
Lemma 8 in \cite{KK1}), which implies $\mbox{Cliff}(\cm
)=d-2\alpha\ge \frac{g-3}{3}$. If $l\ge 3$, then by Lemma 9 in
\cite{KK1} we have $\alpha\le \frac{d+l}{l+1}$, and so
$$\mbox{Cliff}(\cm )\ge
d-\frac{2d+2l}{l+1}=\frac{(l-1)d-2l}{l+1}>\frac{2(l-1)}{(l+1)^2}g-1$$
for $d>\frac{2g}{l+1}+1$. Thus the result follows. \end{proof}
 To prove our main results, we will use a figure which draws the
correspondence between points of $C$ and $C'$ for a multiple
covering morphism $\phi : C\rightarrow C'$. By such a figure, some
computations about line bundles will be simplified if those line
bundles are composed with $\phi$. To do such a work, we need the
following.
\begin{lem}
Assume that $C$  admits
  a simple  $n$-fold covering morphism $\phi :C \rightarrow C'$
 for a smooth curve $C'$ of genus $p$.  And let $\cm$ be a
line bundle on $C$ with $h^0 (C, \cm ) \ge 3$ and $\cli (\cm )
<\frac{g-np}{n-1} -3$. Then  $\cm (-B)$ is either simple or
composed with $\phi$, where $B$ is the base locus of $\cm$.
 \label{Prop:prop
2.6}
\end{lem}

\begin{proof}
 The condition $\cli
(\cm ) <\frac{g-np}{n-1} -3$ also implies  $\cli (\cm (-B))
<\frac{g-np}{n-1} -3$. Thus we may assume $\cm$ is generated by
its global sections. Suppose $\cm$ is neither simple nor composed
with $\phi$. Set $d:=\deg \cm$, $\alpha :=h^0 (C, \cm ) -1$ and
$m:= \deg \varphi _{\cm}$.

Consider a birational projection $\pi$ : $\varphi_{\cm}(C)
\rightarrow \mathbb P^2$ from  general $(\alpha-2)$-points $\sum
^{\alpha -2}_{i=1} Q_i$ of $\varphi_{\cm}(C)$. Then the morphism
$\pi \circ \varphi _{\cm} : C \rightarrow \mathbb P^2$ is
associated to the line bundle $\cm(-\sum ^{\alpha
-2}_{i=1}\varphi_{\cm}^{*} (Q_i))$. Thus we have the following
commutative diagram.

\begin{picture}(300,124)

\put(90,75){$C$}

\put(110,80){\vector(1,0){60}}

\put(100,65){\vector(2,-1){80}}

\put(180,75){$\varphi_{\cm}(C)\subset \mathbb P^{\alpha}$}

\put(200,65){\vector(0,-1){30}}

\put(210,45){$\pi$}

\put(130,87){$\varphi_{\cm}$}

\put(70,35){$\varphi_{\cm(-\sum ^{\alpha
-2}_{i=1}\varphi_{\cm}^{*} (Q_i))}$}

\put(196,15){$\mathbb P^{2}$}

\end{picture}

\noindent{If }$\varphi_{\cm(-\sum ^{\alpha
-2}_{i=1}\varphi_{\cm}^{*} ( Q_i))}$ is  composed with $\phi$,
then $\varphi_{\cm(-\sum ^{\alpha -2}_{i=1}\varphi_{\cm}^{*}  (
Q_i))}=\mu \circ \phi $ for a morphism $\mu$ of degree $\ge 2$.
Then there is a rational morphism $\nu : C' \rightarrow
\varphi_{\cm} (C)$ such that the following diagram commutes as
rational morphisms, since $\pi$ is birational.

\begin{picture}(320,128)

\put(10,75){$C$}

\put(30,80){\vector(1,0){60}}

\put(20,65){\vector(2,-1){80}}

\put(100,75){$\varphi_{\cm}(C)$}

\put(120,30){\vector(0,1){30}}

\put(127,45){$\nu$}

\put(50,85){ $\varphi_{\cm}$}

\put(165,84){ $\pi$}

\put(50,35){$\phi$}

\put(175,33){$\mu$}

\put(116,15){$C'$}

\put(140,80){\vector(1,0){60}}

\put(210,80){$\varphi_{\cm(-\sum ^{\alpha
-2}_{i=1}\varphi_{\cm}^{*} (Q_i))}(C)$}
 \put(140,25){\vector(2,1){80}}

\end{picture}

The smoothness of $C'$ implies that the rational map $\nu$ is
regular, which contradicts that $\varphi_{\cm}$ is not composed
with $\phi$. Accordingly, the morphism $\varphi_{\cm(-\sum
^{\alpha -2}_{i=1}\varphi_{\cm}^{*} ( Q_i))}$ is not  composed
with $\phi$, whence for a general subseries $g^1_{d-m\alpha +2m}$
of $|M(-\sum ^{\alpha -2}_{i=1}\varphi_{\cm}^{*} (Q_i))|$  the
product morphism $\phi \times \varphi _{g^1_{d-m\alpha +2m}}$ is
birational since $\phi$ is simple. Applying the Castelnuovo-Severi
inequality, we obtain $ g\le (n-1)(d-m\alpha +2m-1) +np$ and hence
 $$\cli (\cm )=d-2\alpha \ge \frac{g-np}{n-1} -2m
+1+ (m-2)\alpha \ge \frac{g-np}{n-1} -3$$
 since $\alpha \ge 2$, $m \ge 2$. It
contradicts to  $\cli (\cl )<\frac{g-np}{n-1} -3$, since $\cli
(\ca )\le \cli (\cl )$. Thus the result follows.
\end{proof}

\section{Normal generation of nonspecial line bundles
 on multiple coverings}

In this section, we investigate the normal generation of
nonspecial line bundles on multiple coverings. To do this, we
consider a concrete description for  nonspecial line bundles on a
smooth curve. Let $\cl$ be a nonspecial line bundle on a smooth
curve $C$. There exists a divisor $E>0$ such that $h^0 (C,\ck _C
\otimes \cl ^{-1} (E))=1$ and $h^0 (C,\ck _C \otimes \cl ^{-1}
(E'))=0$ for $E' < E$. Then $\cl\sim \ck _C-g^0_d+E$ for a $g^0_d$
 on $C$ satisfying  ${\rm supp}( g^0 _d) \cap
{\rm supp}(E)=\emptyset$, where $g^0_d$ means a degree $d$ divisor
with $h^0 (C, g^0_d)=1$. Note that we have $\deg \ck _C \otimes
\cl ^{-1} (E) \le g$ since $h^0 (C,\ck _C \otimes \cl ^{-1}
(E))=1$.

Suppose $C$ admits a multiple covering morphism $\phi :
C\rightarrow C'$ for some smooth curve $C'$. Then
$g^0_d=\phi^{*}g^0_t +B$ for $B\ge 0$ on $C$ and $g^0_t$ on $C'$,
where there is no $Q\in C'$ such that $B\ge \phi^{*} Q$. Thus a
nonspecial line bundle $\cl$ on the multiple covering $C$  can be
written by $$\cl\sim \ck _C-\phi^{*}g^0_t -B+E$$ for some $B\ge 0,
E>0$ on $C$ and $g^0 _t$ on $C'$ such that

\noindent(1) $h^0 (C, \phi^{*} g^0 _t +B)=1$,~ (2) ${\rm
supp}(\phi^{*} g^0 _t +B) \cap {\rm supp}(E)=\emptyset$,

 \noindent(3)
$B\ngeq \phi^{*} Q$ for any  $Q \in C'$.

 Note that the nonspecial line bundle $\cl$
with $h^0 (C, \cl ) \ge 3$ is very ample if and only if $\deg E\ge
3$. Using this description, we obtain a sufficient condition for
the normal generation of $\cl$ in terms of $B$ and $E$.

\begin{thm}
Assume that $C$  admits
  a simple  $n$-fold covering morphism $\phi :C \rightarrow C'$
 for a smooth curve $C'$ of genus $p$ with  $g>np$.
 Let $\cl$ be a nonspecial line bundle with  $h^0 (C, \cl ) \ge 3$. We
may set $\cl\sim \ck _C-\phi^{*}g^0_t -B+E$ for some $B \ge 0, ~E
> 0$ on $C$ and $g^0 _t$ on $C'$ satisfying
the  conditions (1), (2) and (3) in the above. Assume  $B = \sum
_{i=1}^{b}\phi^{*} (  Q_i)-P_i$ with $\phi(P_i )=Q_i \in C'$.
Then,
 $\cl$ is normally generated ~if $\deg E> b+2$ and
$\deg \cl > 2g +1-\delta$, which is equivalent to
 $\cli (\cl ) < \delta -1$,~where
$\delta := \min \{ \frac{g}{6}, ~ \frac{g-np}{n-1} -2, ~
\frac{2(2n+\mu -3)}{(2n+\mu
    -1)^2}g~\}$ and  $\mu:=[\frac {2n(n-1)p}{g-np}]$.
\label{thm:theorem7}
\end{thm}

\begin{proof}
  The line bundle
  $\cl$ is very ample  for $\deg E\ge 3$. Suppose $\cl$ fails to be normally
generated.  Then, $C$ has a line bundle $\ca\simeq \cl (-R)$ with
$R>0$, satisfying the conditions in Lemma \ref{rmk:remark2.2}. Let
$\cm:=\ck _C \otimes \ca ^{-1} (-\widetilde{B})$, where
$\widetilde{B}$ is the base locus of $\ck _C  \otimes \ca^{-1}$.
 Set $d:=\deg \cm$ and $\alpha :=h^0
(C, \cm ) -1$. Then  $\alpha \ge 1$ by Lemma \ref{rmk:remark2.2}
\hbox{\rm (iii)}. Assume $\alpha =1$ and $\cm$ is not composed
with  $\phi$. Then by the Castelnuovo-Severi inequality we obtain
$g\le (n-1)(d-1) +np$, since $\phi$ is simple. Then
$$\frac{g-np}{n-1} -1 \le d-2 = \cli (\cm ) \le \cli (\cl ) <
\frac{g-np}{n-1} -3,$$ which cannot occur. Hence $\varphi _{\cm}$
must be composed with the covering morphism $\phi$. Consider the
other cases $\alpha \ge 2$. Assume $\varphi _{\cm}$ is birational.
If $\deg \cm \ge g-1$, then  Lemma \ref{Prop:prop 1.7} and
 Lemma \ref{rmk:remark2.2} (ii) yield
 $$\cli (\cl ) \ge \cli (\cm )\ge \frac{\deg \ck _C \otimes \cm ^{-1}}{3}
 \ge \frac{\deg \ca}{3} \ge \frac{g-1}{6},$$ which cannot occur.
Accordingly, $\deg \cm\le g-1$.

 By the  birationality of $\varphi _{\cm}$ and the
Castelnuovo-Severi inequality, we have $g\le (n-1)(d-1)+np$ and so
$(\frac{g-np}{g})\frac{2g}{d-1}\le 2(n-1)$. Since $g>np$,
$$[\frac{2g}{d-1}] \le [2(n-1)(1+ \frac{np}{g-np} )]\le 2(n-1) +\mu ,$$
where  $\mu:=[\frac {2n(n-1)p}{g-np}]$. Lemma \ref{Prop:prop 1.7}
implies either  $\cli (\cm)
>\frac{2(2n+\mu -3)}{(2n+\mu -1)^2}g -1$ or $\cli (\cl ) \ge \frac{g}{3}
-1$, which is a contradiction to $\cli (\cl ) < \delta -1$. Thus
we have  $m\ge 2$, and hence
 $\cm=\ck _C \otimes
\ca^{-1} (-\widetilde{B})$ is composed with the covering morphism
$\phi$ by Lemma \ref{Prop:prop 2.6} and the condition $\cli (\cl )
<\frac{g-np}{n-1} -3$.

 Note that $R$ contains $E$, since $h^0 (C, \ck
_C \otimes \ca^{-1})\ge 2$, $\co _C (\phi ^{*} g^0_t +B -E)=\ck _C
\otimes \cl ^{-1} =\ck  _C \otimes \ca^{-1} (-R)$, $h^0 (C,
\phi^{*} g^0 _t +B)=1$ and ${\rm supp}(\phi^{*} g^0 _t +B) \cap
{\rm supp}(E)=\emptyset$. Set $R(-E)= \phi^{*} (F_l )  + R_0$
 for a divisor $R_0 \ge 0$  on $C$ and  a degree $l$ divisor
  $F_l$ on $C'$ such that
  $R_0 \ngeq \phi^{*} Q $ for any $Q\in C'$.
Assume only the points $P_1$,...,$P_k$ of $\sum _{i=1}^{b}P_i$
    are  contained in $R$. Set $G:=F_l +\sum _{i=
    1}^{k} \phi (P_i)$. Then $\phi ^* (g^0_t +G)$ corresponds
    to the pullback part of $\ck \otimes \ca ^{-1}$ via the
    covering    morphism $\phi$.

 For a better understanding, consider the
following Figure \ref{figure 5} which figures the correspondence
of  points on curves $C$ and $C'$.

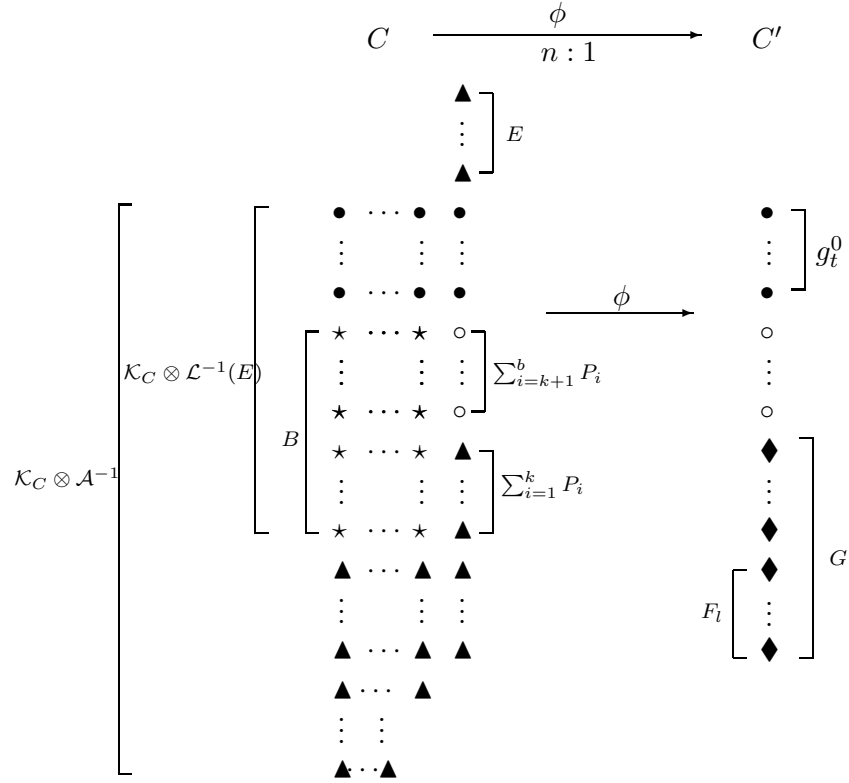
\begin{figure}[ht]

\begin{picture}(300,300)

\put(153,275){$C$}

\put(178,280){\vector(1,0){100}}

\put(221,285){$\phi$}

\put(218,270){$n:1$}

\put(297,275){$C'$}

\put(140,165){$\star$}

\put(142,148){$\vdots$}

\put(140,135){$\star$}

\put(153,165){$\cdots$}

\put(153,135){$\cdots$}

\put(170,165){$\star$}

\put(172,148){$\vdots$}

\put(170,135){$\star$}

\put(185,255){$\blt$}

\put(187,238){$\vdots$}

\put(185,225){$\blt$}

\put(185,210){$\bullet$}

\put(187,193){$\vdots$}

\put(185,180){$\bullet$}

\put(140,210){$\bullet$}

\put(142,193){$\vdots$}

\put(140,180){$\bullet$}

\put(142,148){$\vdots$}

\put(140,135){$\star$}

\put(140,120){$\star$}

\put(142,103){$\vdots$}

 \put(140,90){$\star$}

\put(153,210){$\cdots$}

\put(153,180){$\cdots$}

\put(153,165){$\cdots$}

\put(153,165){$\cdots$}

 \put(153,135){$\cdots$}

\put(153,120){$\cdots$}

\put(153,90){$\cdots$}

 \put(170,210){$\bullet$}

\put(172,193){$\vdots$}

\put(170,180){$\bullet$}

\put(170,165){$\star$}

\put(172,148){$\vdots$}

\put(170,135){$\star$}

\put(170,120){$\star$}

\put(172,103){$\vdots$}

\put(170,90){$\star$}

\put(185,165){$\circ$}

\put(187,148){$\vdots$}

\put(185,135){$\circ$}

\put(185,120){$\blt$}

\put(187,103){$\vdots$}

\put(185,90){$\blt$}

\put(130,168){\line(1,0){5}}

\put(130,168){\line(0,-1){76}}

\put(130,92){\line(1,0){5}}

{\scriptsize \put(121,125){$B$}}

\put(111,215){\line(1,0){5}}

\put(111,215){\line(0,-1){123}}

\put(111,92){\line(1,0){5}}

{\scriptsize \put(62,150){$\ck _C \otimes \cl ^{-1} (E)$}}

\put(197,168){\line(-1,0){5}}

\put(197,168){\line(0,-1){30}}

\put(197,138){\line(-1,0){5}}

{\scriptsize \put(200,150){$\sum _{i=k+1}^{b} P_i $}}

\put(140,75){$\blt$}

\put(142,58){$\vdots$}

\put(140,45){$\blt$}

\put(153,75){$\cdots$}

\put(153,45){$\cdots$}

\put(170,75){$\blt$}

\put(172,58){$\vdots$}

\put(170,45){$\blt$}

\put(185,75){$\blt$}

\put(187,58){$\vdots$}

\put(185,45){$\blt$}

\put(140,30){$\blt$}

\put(142,13){$\vdots$}

\put(140,0){$\blt$}

\put(150,30){$\cdots$}

\put(170,30){$\blt$}

\put(145,0){$\cdots$}

\put(157,0){$\blt$}

\put(157,13){$\vdots$}

\put(200,123){\line(-1,0){5}}

\put(200,123){\line(0,-1){31}}

\put(200,92){\line(-1,0){5}}

{\scriptsize\put(203,108){$\sum _{i=1}^{k} P_i $}}

\put(60,216){\line(1,0){5}}

\put(60,216){\line(0,-1){215}}

\put(60,1){\line(1,0){5}}

{\scriptsize \put(21,110){$\ck  _C  \otimes \ca^{-1}$}}

\put(300,210){$\bullet$}

\put(302,193){$\vdots$}

\put(300,180){$\bullet$}

\put(300,165){$\circ$}

\put(302,148){$\vdots$}

\put(300,135){$\circ$}

\put(300,120){$\bltd$}

\put(302,103){$\vdots$}

\put(300,90){$\bltd$}

\put(220,175){\vector(1,0){55}}

\put(300,75){$\bltd$}

\put(302,57){$\vdots$}

\put(300,45){$\bltd$}

\put(290,78){\line(1,0){5}}

\put(290,78){\line(0,-1){33}}

\put(290,45){\line(1,0){5}}

{\scriptsize \put(278,60){$F_l$}}

\put(320,128){\line(-1,0){5}}

\put(320,128){\line(0,-1){83}}

\put(320,45){\line(-1,0){5}}

{\scriptsize \put(326,80){$G$}}

\put(200,258){\line(-1,0){5}}

\put(200,258){\line(0,-1){30}}

\put(200,228){\line(-1,0){5}}

{\scriptsize \put(205,240){$E$}}

\put(317,214){\line(-1,0){5}}

\put(317,214){\line(0,-1){30}}

\put(317,184){\line(-1,0){5}}

\put(321,195){$g^0_t$}

\put(245,178){$\phi$}

\end{picture}
\caption{$\cl$ is nonspecial \label{figure 5}}

\end{figure}

Here,

\noindent i) $E$ : the sum of  the points being arranged as
triangles on the  left upper side,

\noindent ii) $\ck _C \otimes \cl ^{-1}(E)$ : the sum of the
points being arranged as black dots and stars on the left side,

\noindent iii) $B$ : the sum of the points being arranged as the
assigned  stars on the left side,

\noindent iv) $R$ : the sum of  the points being arranged as
triangles on the left  side,

\noindent v) $\sum _{i=1}^{k} P_i $ : the sum of the points being
 arranged as the assigned triangles on the left side,

\noindent vi)$\sum _{i=k+1}^{b} P_i $ : the sum of the points
being
 arranged as blank circles on the left side,

 \noindent vii) $\ck  _C  \otimes \ca^{-1}$ : the sum of the points
being arranged as black dots and
 stars  on the left side and triangles on the left lower side,

\noindent viii) $g^0_{t}$ : the sum of the points being arranged
as black dots on the  right side,

\noindent ix) $G$ : the sum of the  points being arranged as $(l+k
)$-black diamonds  on the right side,

 \noindent x)  $F_l$ : the sum of the points
being arranged as the assigned $l$-black diamonds on the right
side.\\

From Figure \ref{figure 5}, we easily see that $h^0 (C, \ck  _C
\otimes \ca^{-1})\le k +l +1$, since $\varphi _{\cm}$ is composed
with $\phi$. Thus
$$\deg R \le 2(h^0 (C,\ck _C \otimes
\ca^{-1})-h^0 (C,\ck _C \otimes \cl ^{-1}))\le 2(k +l +1),$$ since
$\mbox{Cliff}(\ck  _C \otimes \ca^{-1})\le \mbox{Cliff}(\ck _C
\otimes \cl ^{-1})$ and $\ca\cong \cl (-R)$. Note that $\deg R \ge
\deg E +nl + k$, since $R(-E) \ge \phi^{*} F_l + \sum _{i=1}^{k}
P_i $. Accordingly,
 $\deg E \le k+2$, which cannot occur for $\deg E >b+2$. Thus the
theorem is proved.
\end{proof}

Specifically, the theorem is more simplified for double coverings.

\begin{cor} Let $C$ admit a double covering
   morphism $\phi : C  \rightarrow C'$ for a smooth curve $C'$ of genus $p$
with $3g >8(p+1)$. And let $\cl$ be a
 nonspecial  line bundle with $h^0 (C, \cl ) \ge 3$ and  $
 ~\deg \cl >2g+1-\frac{g}{6}$.
 We may set $\cl\sim \ck _C-\phi^{*}g^0_t-B+E$ as before.
   Then $\cl$ is normally
generated~ if $\deg E > \deg B+2$. \label{thm:theorem14}
\end{cor}
\begin{proof} Since $n=2$ and $3g
>8(p+1) $,  we obtain $\frac{g-np}{n-1} -2 \ge \frac{g}{6} $ and
 $\mu:=[\frac {2n(n-1)p}{g-np}]\le 6$, whence
$\frac{2(2n+\mu -3)}{(2n+\mu -1)^2}> \frac{1}{6}$. Thus  the
result follows from Theorem \ref{thm:theorem6}.
\end{proof}

 It has known that if  $C$  is a double
covering of a smooth curve $C'$ of genus $p$ then $C$ admits no
nonspecial normally generated line bundles $\cl$ with $g+5 \le
\deg \cl \le 2g-3p$(\cite{LM1}). On the one hand, we obtain the
following result by Theorem \ref{thm:theorem7}.

\begin{cor}
Assume that $C$  admits
  a simple  $n$-fold covering morphism $\phi :C \rightarrow C'$
 for a smooth curve $C'$ of genus $p$ with $g\ge
\mbox{max}\{n^2 p , \frac{(2n+1)^2}{3} \}$. Then $C$ has a
nonspecial normally generated line bundle  of degree $d$ for any
$$ d\ >\ \begin{cases} \mbox{max}\{ 2g-np,~ 2g+1-\frac{g}{6}
\},~\mbox{ if }~n\le 4\\
 \mbox{max}\{ 2g-np,~2g+1-\frac{2(2n-1)}{(2n+1)^2}g\},~\mbox{ if } ~n\ge
 5
\end{cases}$$
\label{thm:theorem11}
\end{cor}
\begin{proof}
 Choose a general effective divisor
$D_t$ of degree $ t\le p$ on $C'$, whence $D_t$ is a $g^0_t$. The
Caselnuovo-Severi inequality implies $h^0 (C, \phi ^* D_t ) =1$,
since $g \ge n^2 p$ and $\phi$ is simple. Set $B=0$. Choose a
divisor $E$ on $C$ such that $\deg E >2$ and supp$(\phi ^* D_t )
\cap$supp$(E)= \emptyset$. Then, $\cl:=\ck _C -\phi^{*}  D_t +E$
is a nonspecial very ample line bundle with $h^0 (C, \cl )\ge 4$
since $\deg \cl \ge g+3$. By varying the numbers $\deg E$ and $t$
within $\deg E >2$ and $0 \le t\le p$, we see that $\deg \cl=2g-nt
+\deg E-2$ can take any number $d> 2g -np $.

 From $g\ge n^2 p$, we
get $\mu :=[ \frac{2n(n-1)p}{g-np}]\le 2$ and so $\frac{2(2n+\mu
-3)}{(2n+\mu -1)^2}g  \le \frac{2(2n-1)}{(2n+1)^2}g $. And the
condition $g \ge  n^2 p$ yields $\frac{g-np}{n-1} \ge \frac{g}{n}$
and so $\frac{2(2n-1)}{(2n+1)^2}g  \le \frac{g-np}{n-1}$ for $g\ge
\frac{(2n+1)^2}{3}$. Hence, $\cl$ is normally generated if $\deg
\cl > \mbox{max}\{2g +1-\frac{2(2n-1)}{(2n+1)^2}g, 2g+1-
\frac{g}{6} \}$ by Theorem \ref{thm:theorem7}. As a consequence,
 $C$ admits a nonspecial normally generated line bundle of degree $d$
for any
 $d> \mbox{max}\{ 2g-np,~ 2g+1-\frac{g}{6},
~2g+1-\frac{2(2n-1)}{(2n+1)^2}g \}$.   Accordingly,  the result
holds since $\frac{g}{6} \le \frac{2(2n-1)}{(2n+1)^2}g $ for $n\le
4$ and $\frac{2(2n-1)}{(2n+1)^2}g  \le \frac{g}{6}$ for $n\ge 5$.
\end{proof}
The condition $\deg E>b+2$ in Theorem \ref{thm:theorem7} is
optimal, since
 there exists  a nonspecial very ample line
 bundle  failing to be normally generated with $\deg E = b+2$ as follows.

\begin{exm} ~Assume that $C$  admits
  a simple  $n$-fold covering morphism $\phi :C \rightarrow C'$
 for a smooth  non-rational curve   $C'$. Take a base point free
 complete $g^1_d$ on $C'$. Set $g^0_{d-1}:=g^1_d -Q$ for a point  $Q\in
C'$ and $B:=\phi^{*}(Q) -P$ for $P\in \phi^{*} Q$. Choose a
divisor $E>0$ on $C$ satisfying deg$E=3$ and ${\rm supp}(B
+\phi^{*} g^0 _{d-1} ) \cap {\rm supp}(E)=\emptyset$.  Then $\cl
:= \ck _C -\phi^{*} g^0_{d-1} -B+E$ is  very ample and fails to be
normally generated if $g>(n-1)(nd-2) +ng(C')$.
\label{exm4.5}\end{exm}
\begin{proof} The condition  $g>(n-1)(nd-2)+ng(C')$ with $g(C')
>0$ yields $\deg \cl=2g+2 -nd > g+2$, whence $h^0 (C,\cl )\ge 3$.
 According to the Caselnuovo-Severi inequality, we have
 $h^0 (C,\phi ^* g^0_{d-1} +B )=1$, since $\phi$ is simple with $g>(n-1)(nd-2) +ng(C')$.
Thus $\cl$ is very ample for $\deg E = 3$ as mentioned before.

For a better understanding, we figure the correspondence of
 points on curves $C$ and $C'$
in Figure \ref{figure 6}.
 Here,

 \noindent i) $E$ : the sum of the
points being arranged as triangles on the left upper side,

\noindent ii) $\ck _C \otimes \cl ^{-1}(E)$ : the sum of the
points being arranged as black dots and stars on the left side,

\noindent iii) $B$ : the sum of the points being arranged as stars
on the left bottom line,

\noindent iv) $g^1_d$ : the sum of the points being arranged as
black dots and black diamonds on the  right side,

\noindent v) $g^0_{d-1}$ : the sum of the points being arranged as
black dots on the right side.\\

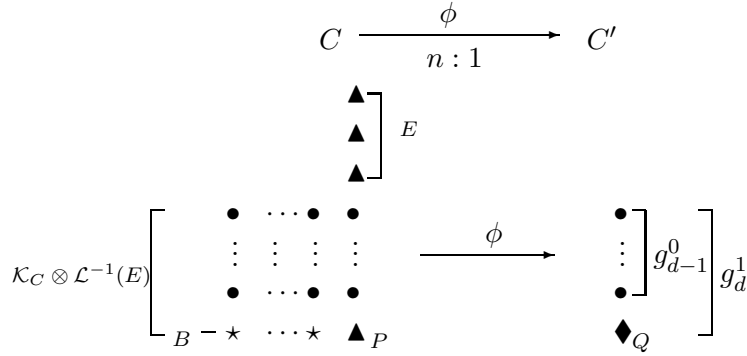
\begin{figure}

\begin{picture}(360,150)

\put(145,130){$C$}

\put(160,135){\vector(1,0){75}}

\put(190,140){$\phi$}

\put(185,123){$n:1$}

\put(245,130){$C'$}

\put(110,65){$\bullet$}

\put(112,48){$\vdots$}

\put(110,35){$\bullet$}

\put(110,20){$\star$}

\put(125,65){$\cdots$}

\put(127,48){$\vdots$}

\put(125,35){$\cdots$}

\put(125,20){$\cdots$}

\put(140,65){$\bullet$}

\put(142,48){$\vdots$}

\put(140,35){$\bullet$}

\put(140,20){$\star$}

\put(155,110){$\blt$}

\put(155,95){$\blt$}

\put(155,80){$\blt$}

\put(168,113){\line(-1,0){5}}

\put(168,113){\line(0,-1){32}}

\put(168,81){\line(-1,0){5}}

{\scriptsize \put(175,97){$E$}}

\put(155,65){$\bullet$}

\put(157,48){$\vdots$}

\put(155,35){$\bullet$}

\put(155,20){$\blt$}

{\scriptsize \put(164,17){$P$}}

\put(101,23){\line(1,0){5}}

{\scriptsize \put(90,18){$B$}}

\put(82,69){\line(1,0){5}}

\put(82,69){\line(0,-1){47}}

\put(82,22){\line(1,0){5}}

 {\scriptsize \put(30,42){$\ck _C \otimes \cl ^{-1}(E)$}}

\put(255,65){$\bullet$}

\put(257,48){$\vdots$}

\put(255,35){$\bullet$}

\put(255,20){$\bltd$}

{\scriptsize \put(262,17){$Q$}}

\put(267,69){\line(-1,0){5}}

\put(267,69){\line(0,-1){32}}

\put(267,37){\line(-1,0){5}}

\put(270,49){$g^0_{d-1}$}

\put(292,69){\line(-1,0){5}}

\put(292,69){\line(0,-1){47}}

\put(292,22){\line(-1,0){5}}

\put(295,42){$g^1_d$}

\put(183,52){\vector(1,0){50}}

\put(207,57){$\phi$}

\end{picture}

\caption{ Example of nonspecial line bundle\label{figure 6}}

\end{figure}

Set $D_4 := E+P$. We see that $h^0 (C,\ck _C \otimes \cl
^{-1}(D_4))\ge 2$, which yields $\dim \overline{\phi_{\cl
}(D_4)}\le 1$.  Thus $\cl$ fails to be normally generated by Lemma
\ref{Prop:prop 1.6}, since $h^1 (C, \cl ^2 (-D_4 ))=0$ for $\deg
\cl \ge g+2$.
\end{proof}

Using Lemma \ref{Prop:prop 1.6}, we obtain a sufficient condition
for  nonspecial line bundles failing to be normally generated.

\begin{thm}
 Assume that $C$  admits
  a multiple covering morphism $\phi :C \rightarrow C'$
 for a smooth curve $C'$. And let
$\cl$ be a nonspecial line bundle with $h^0 (C,\cl ) \ge 3$. We
may set $\cl\sim \ck _C-\phi^{*} g^0_t -B+E$ as in  Theorem
 \ref{thm:theorem7}. Assume  $B = \sum _{i=1}^{b}\phi^{*} ( Q_i)-P_i$,
  $\phi(P_i )=Q_i$.
  Then  $\cl$ is very ample and  fails to be  normally generated
   if $a\ge 3 $,~$a+b  >\frac{(a+b-r)(a+b-r-1)}{2}$ and
   $h^1 (C, \cl ^2 (-E-\sum ^{b} _{i=1 } P_i  ))=0$,~ where
$a:=\deg E$, $r:=h^0 (C', g^0_t +\sum ^{b} _{i=1 }  Q_i )-1$.
\label{Prop:prop 5.2}
\end{thm}

\begin{proof} $\cl$ is  very ample for $\deg E \ge 3$.
Let  $D :=E +\sum ^{b} _{i=1 } P_i $. Since $h^0 (C,\ck _C \otimes
\cl ^{-1} (D))\ge h^0 (C', g^0_t +\sum ^{b} _{i=1 }  Q_i) $, the
Riemann-Roch Theorem gives $h^0 (C, \cl ) -h^0 (C, \cl (-D))\le
a+b-r-1$ and so $\dim \overline{\phi_{\cl }(D)}\le a+b-r-2$.
Accordingly, $\cl$ fails to be normally generated by Lemma
\ref{Prop:prop 1.6}, if $\deg D=a+b
>\frac{(a+b-r)(a+b-r-1)}{2}$ and
$h^1 (C, \cl ^2 (-E-\sum ^{b} _{i=1 } P_i  ))=0$.
\end{proof}

From Theorem \ref{Prop:prop 5.2}, we explicitly obtain nonspecial
 line bundles failing to be noramlly generated on  multiple
coverings as the following.

\begin{rmk}\label{rmk:remark4.7} ~Let $C$  admit
  a simple  $n$-fold covering morphism $\phi :C \rightarrow C'$
 for a smooth  non-rational curve   $C'$. Let $g^r_d$ be a complete linear series  on $C'$ with $r\ge
1$. Assume $g>(n-1)(nd-r-1) +ng(C')$. For a general
 $\sum ^{r} _{i=1} Q_i$ on $C'$,  we have dim$g^r_d
(-\sum ^{r} _{i=1} Q_i )=0$. Set $g^0 _{d-r} :=g^r_d (-\sum ^{r}
_{i=1} Q_i )$ and $B:=\sum _{i=1}^{r}\phi^{*} ( Q_i)-P_i$ with
$\phi (P_i) =Q_i$. Choose a divisor $E> 0$ on $C$ such that ${\rm
supp}(\phi^{*} g^0 _{d-r} + B) \cap {\rm supp}(E)=\emptyset$,
$\deg E \ge 3$ and $r> \frac{\deg E(\deg E -3)}{2}$. Let $\cl:=\ck
_C -\phi ^{*} g^0_{d-r} -B +E$. Note that the assumption
$g>(n-1)(nd-r-1) +ng(C')$ yields $h^0 (C, \phi^{*} g^0 _{d-r}
+B)=1$ by the Castelnuovo-Severi inequality. And the assumption
also gives $\deg \cl \ge g+2$ and so $h^0 (C,\cl ) \ge 3$.
   According to Theorem
\ref{Prop:prop 5.2}, $\cl$ is very ample and fails to be normally
generated if $h^1 (C, \cl ^2 (-E-\sum ^{r} _{i=1 } P_i  ))=0$.
\end{rmk}

\section{Normal generation of special line bundles on multiple coverings}

In this section, we investigate the normal generation of special
line bundles on a multiple covering. Firstly, we give a sufficient
condition for a special line bundle being normally generated.

\begin{thm}
Assume that $C$  admits
  a simple  $n$-fold covering morphism $\phi :C \rightarrow C'$
 for a smooth curve $C'$ of genus $p$ with $g >np$. Let $\cl$
be a special very ample line bundle on $C$ with $\deg \cl
>\frac{3g-3}{2}$.
 Assume $\ck _C \otimes \cl ^{-1}=\phi ^{*} \cn (-\sum ^{b} _{i=1 }P_i)$
  for a line bundle $\cn$ on $C'$ and  $\sum ^{b} _{i=1 }P_i$
 on $C$ such that  $\sum ^{b} _{i=1 } \phi (P_i) \leq \cn$
  and $P_i+ P_j  \nleqslant \phi^{*} (\phi
 (P_i))$.  Then,  $\cl$ is normally generated
 if $b\le 3$ and $\deg \cl > 2g +1-2h^1 (C, \cl )-\delta$,
 which is equivalent to
 $\cli (\cl ) < \delta -1$,~where
$\delta := \min \{ \frac{g}{6}, ~ \frac{g-np}{n-1} -2, ~
\frac{2(2n+\mu -3)}{(2n+\mu
    -1)^2}g~\}$ and  $\mu:=[\frac {2n(n-1)p}{g-np}]$.

\label{thm:theorem6}
\end{thm}

\begin{proof}
Suppose $\cl$ fails to be normally generated. Then,  we have
$\ca\simeq \cl (-R)$, $R>0$,  and $\cm:=\ck  _C \otimes \ca^{-1}
(-\widetilde{B})$ as in the proof of Theorem \ref{thm:theorem7}.
Set $d:=\deg \cm$, $\alpha :=h^0 (C, \cm ) -1$ and $m:= \deg
\varphi _{\cm}$.
 Note that $\alpha \ge 2$ by Lemma \ref{rmk:remark2.2}
\hbox{\rm (iii)}.
 Then, $\varphi _{\cm}$ is composed with $\phi$ by the same reason as in the proof of Theorem
\ref{thm:theorem7}.

 Accordingly, $\ck _C \otimes \cl ^{-1}(-B)$
is also composed with
 $\phi$ since $\ck  _C  \otimes \ca^{-1}
>\ck _C \otimes \cl^{-1}$ and $h^0 (C,\ck  _C  \otimes \ca^{-1} ) > h^0 (C,\ck  _C  \otimes \cl^{-1} )$,
  where $B$ is the base locus of $\ck _C \otimes \cl ^{-1}$. Thus
 $B \ge \sum _{i=1}^{b} (\phi^{*} (\phi (P_i )) -P_i )$, since $\ck _C \otimes
\cl^{-1} =\phi^{*} N(-\sum _{i=1}^{b} P_i)$. Note that $$\ck _C
\otimes \cl ^{-1}(-\sum _{i=1}^{b}(\phi^{*} (\phi (P_i )) -P_i
))=\phi ^{*} (\cn (-\sum _{i=1}^{b} \phi (P_i ) ).$$ Then the
pull-back part of $\ck  _C  \otimes \ca^{-1}$ via $\phi$ becomes
$\phi ^*  (\cn(-\sum _{i=1}^{b} \phi (P_i )+G))$ for some divisor
$G>0$ on $C'$, since $\ck  _C  \otimes \ca^{-1}
>\ck _C \otimes \cl^{-1}$.
 Set $R= \phi^{*} (F_l ) + R_0$
 for a divisor $R_0 \ge 0$  on $C$ and  a degree $l$ divisor
  $F_l$ on $C'$ such that
  $R_0 \ngeqslant \phi^{*} Q$ for any  $Q\in C'$. Because of
   $\ca\simeq \cl (-R)$,
  we have $G=F_l +\sum _{i=1}^{k}\phi (P_i)$,
 where only the points $P_1$,...,$P_k$  of $\sum
_{i=1}^{b}P_i$ are  contained in $R$.

  For a better understanding, we figure the correspondence of
points on curves $C$ and $C'$ in Figure \ref{figure 2}: Here,\\

\noindent i) $\ck _C \otimes \cl ^{-1}$ : the sum of the points
being arranged as black dots and stars on the left side,

\noindent ii) $B$ : the sum of the points being arranged as stars
 on the left side,

\noindent iii) $R$ : the sum of  the points being arranged as
triangles on the left side,

\noindent iv) $\ck  _C  \otimes \ca^{-1}$ : the sum of the points
being arranged as black dots, stars and triangles on the left
side,

\noindent v) $G$ : the sum of the  points being arranged as the
assigned $(l+k )$-black diamonds  on the right side,

\noindent vi) $\cn(-\sum _{i=1}^{b} \phi (P_i ))$ : the sum of the
points being arranged as black dots on the right side,

\noindent vii)  $F_l$ : the sum of the points being arranged as
the assigned $l$-black diamonds on the right side,

 \noindent viii) $\sum _{i=k+1}^{b} P_i $: the sum of the points being
 arranged as blank circles on the left side,

\noindent ix) $\sum _{i=1}^{k} P_i $ : the sum of the points being
 arranged as the assigned triangles on the left side,

\noindent x) the points of $\sum _{i=1}^{b} P_i $ are arranged in
one column since $P_i+ P_j  \nleqslant \phi^{*} (\phi
 (P_i))$.\\

\begin{figure}[ht]

\begin{picture}(360,310)

\put(125,290){$C$}

\put(154,295){\vector(1,0){100}}

\put(195,300){$\phi$}

\put(190,280){$n:1$}

\put(269,290){$C'$}

\put(110,270){$\bullet$}

\put(110,255){$\bullet$}

\put(112,238){$\vdots$}

\put(110,225){$\bullet$}

\put(110,210){$\star$}

\put(112,193){$\vdots$}

\put(110,180){$\star$}

\put(123,270){$\cdots$}

\put(123,255){$\cdots$}

\put(123,238){$\cdots$}

\put(123,225){$\cdots$}

\put(123,210){$\cdots$}

\put(123,180){$\cdots$}

\put(140,270){$\bullet$}

\put(140,255){$\bullet$}

\put(142,238){$\vdots$}

\put(140,225){$\bullet$}

\put(140,210){$\star$}

\put(142,193){$\vdots$}

\put(140,180){$\star$}

\put(155,270){$\bullet$}

\put(155,255){$\bullet$}

\put(157,238){$\vdots$}

\put(155,225){$\bullet$}

\put(155,210){$\star$}

\put(157,193){$\vdots$}

\put(155,180){$\star$}

\put(110,165){$\star$}

\put(112,148){$\vdots$}

\put(110,135){$\star$}

\put(110,120){$\star$}

\put(123,165){$\cdots$}

\put(123,135){$\cdots$}

\put(123,90){$\cdots$}

\put(123,120){$\cdots$}

\put(140,165){$\star$}

\put(140,135){$\star$}

\put(142,148){$\vdots$}

\put(140,120){$\star$}

\put(142,102){$\vdots$}

\put(140,90){$\star$}

\put(155,165){$\circ$}

\put(157,148){$\vdots$}

\put(155,135){$\circ$}

\put(155,120){$\blt$}

\put(157,102){$\vdots$}

\put(155,90){$\blt$}

\put(100,213){\line(1,0){5}}

\put(100,213){\line(0,-1){121}}

\put(100,92){\line(1,0){5}}

{\scriptsize\put(90,155){$B$}}

\put(80,272){\line(1,0){5}}

\put(80,272){\line(0,-1){180}}

\put(80,92){\line(1,0){5}}

{\scriptsize\put(43,177){$\ck _C \otimes \cl ^{-1}$}}

\put(112,102){$\vdots$}

\put(110,90){$\star$}

\put(110,75){$\blt$}

\put(112,58){$\vdots$}

\put(110,45){$\blt$}

\put(123,75){$\cdots$}

\put(123,45){$\cdots$}

\put(140,75){$\blt$}

\put(142,58){$\vdots$}

\put(140,45){$\blt$}

\put(155,75){$\blt$}

\put(157,58){$\vdots$}

\put(155,45){$\blt$}

\put(110,30){$\blt$}

\put(112,13){$\vdots$}

\put(110,0){$\blt$}

%

\put(123,30){$\cdots$}

\put(118,0){$\cdots$}

\put(140,30){$\blt$}

\put(132,13){$\vdots$}

\put(132,0){$\blt$}


%

\put(90,138){\line(1,0){5}}

\put(90,138){\line(0,-1){138}}

\put(90,0){\line(1,0){5}}

{\scriptsize\put(80,67){$R$}}

\put(40,272){\line(1,0){5}}

\put(40,272){\line(0,-1){272}}

\put(40,0){\line(1,0){5}}

{\scriptsize\put(2,100){$\ck  _C  \otimes \ca^{-1}$}}

\put(270,270){$\bullet$}

\put(270,255){$\bullet$}

\put(272,238){$\vdots$}

\put(270,225){$\bullet$}

\put(270,210){$\bullet$}

\put(272,192){$\vdots$}

\put(270,180){$\bullet$}

\put(270,165){$\circ$}

\put(272,148){$\vdots$}

\put(270,135){$\circ$}

\put(270,120){$\bltd$}

\put(272,103){$\vdots$}

\put(270,90){$\bltd$}

\put(188,185){\vector(1,0){60}}

\put(270,75){$\bltd$}

\put(272,57){$\vdots$}

\put(270,45){$\bltd$}

\put(290,272){\line(-1,0){5}}

\put(290,272){\line(0,-1){90}}

\put(290,182){\line(-1,0){5}}

{\scriptsize \put(292,235){$\cn(-\sum _{i=1}^{b}\phi (P_i ))$}}

\put(170,168){\line(-1,0){5}}

\put(170,168){\line(0,-1){30}}

\put(170,138){\line(-1,0){5}}

{\scriptsize \put(175,148) {$\sum _{i=k+1}^{b} P_i $}}

\put(170,93){\line(-1,0){5}}

\put(170,123){\line(0,-1){30}}

\put(170,123){\line(-1,0){5}}

{\scriptsize \put(175,103) {$\sum _{i=1}^{k} P_i $}}

\put(260,78){\line(1,0){5}}

\put(260,78){\line(0,-1){30}}

\put(260,48){\line(1,0){5}}

{\scriptsize\put(245,62){$F_l$}}

\put(290,125){\line(-1,0){5}}

\put(290,125){\line(0,-1){77}}

\put(290,48){\line(-1,0){5}}

{\scriptsize\put(295,86){$G$}}

\put(215,193){$\phi$}

\end{picture}

\caption {$\cl$ is special \label{figure 2} }

\end{figure}

Suppose $k \le 1$. The base point freeness  of $\cl$ implies $h^0
(C, \ck  _C  \otimes \ca^{-1})-h^0 (C, \ck  _C  \otimes \cl ^{-1})
\le l$, since $\ck  _C  \otimes \ca^{-1} (-\widetilde{B})$ is
composed with $\phi$. From the condition $\cli (\ca )\le \cli (\cl
)$ we obtain
$$\frac{nl +k }{2} \le \frac{\deg R}{2}\le h^0 (C, \ck  _C  \otimes \ca^{-1})-h^0 (C, \ck  _C  \otimes \cl ^{-1}) \le l,$$ which yields $n
=2$ and $k =0$. Then $$h^0 (C, \ck  _C  \otimes \ca^{-1})-h^0 (C,
\ck  _C  \otimes \cl ^{-1}) \le l-1,$$ since $\cl$ is very ample
and $\ck  _C  \otimes \ca^{-1} (-\widetilde{B})$ is composed with
$\phi$. It cannot also occur as in the above. Thus we have $k \ge
2$.
 According to  the very ampleness of
$\cl$ and $h^0 (C,\ck _C \otimes \cl ^{-1})=h^0 (C', \cn (-\sum
_{i=1}^{b} \phi (P_i )))$,  $$ h^0 (C,\ck  _C  \otimes \ca^{-1})=
h^0 (C', \cn (-\sum _{i=1}^{b} \phi (P_i ) +G))\le h^0 (C, \ck _C
\otimes \cl ^{-1} )+l +k -2,$$  since $\ck  _C  \otimes \ca^{-1}
(-\widetilde{B})$ is composed with $\phi$. Hence
$$\deg R\le 2(h^0 (C,\ck  _C  \otimes \ca^{-1})-h^0 (C,\ck _C \otimes \cl ^{-1}))\le
2(l+k-2),$$ since $\mbox{Cliff}(\ck  _C  \otimes \ca^{-1})\le
\mbox{Cliff}(\ck _C \otimes \cl ^{-1})$ and $\ca\cong \cl (-R)$.
Then the inequality $\deg R \ge nl+k$ yields $k\ge 4$, which is a
contradiction to $b\le 3$. Thus $\cl$ is normally generated.
\end{proof}

Note that  any $n$-fold covering morphism is simple in case $n$ is
prime. Specifically, if $C$ is a double covering then we have the
following.

\begin{cor} Let $C$ be a double
covering of a smooth curve $C'$ of genus $p$ via a morphism $\phi$
with $3g>8(p+1) $. And let $\cl$ be a special  very ample line
bundle on $C$ with $\ck _C \otimes \cl ^{-1}=\phi^{*}
  \cm \otimes \co _C (B )$ for a   line bundle $\cm$ on $C'$
  and  a divisor $B\ge 0$ on $C$ such that $B \ngeq \phi^* Q$ for any $Q \in C'$.
  If $\deg B \le 3$ and $\deg \cl > {\mbox max} \{
\frac{3g-3}{2},~ 2g+1-2h^1 (C,\cl)-\frac{g}{6}
      \}$, then $\cl$ is normally generated.

\label{thm:theorem6.1}\end{cor}
This corollary can be shown
similarly to  Corollary \ref{thm:theorem14}. On the one hand, the
condition $b\le 3$ in the theorem is sharp in some sense, since
for $ b= 4$ there are special very ample line bundles failing to
be normally generated on   multiple coverings as follows.
\begin{exm} ~Let $C$ be a  simple $n$-fold covering of a smooth
 plane curve $C'$ of degree
$d$ with $g\ge 3ng(C')$ and $d \ge n+2$. Denote the covering
morphism by $\phi : C \rightarrow C'$.  Let $H$ be a general line
section of $C'$. Choose divisors $F_4 : =\sum_{i=1} ^{4} Q_i \le
H$ on $C'$
 and $D_4 :=\sum_{i=1} ^{4} P_i$
on $C$ such that $\phi (P_i) =Q_i$. Let $\cl$ be the line bundle
with $\ck _C \otimes \cl ^{-1}=\phi ^{*} ( {\cal O}_{C'} (H))(-D_4
)$.  Then $\cl$ is very ample and fails to be normally generated.
\end{exm}

\begin{proof} Suppose that we have $h^1 (C, \cl ) \ge 2$
or $\cl$ is not very ample. Then $h^0 (C,\phi ^* {\cal
O}_{C'}(H)(-D_4 +P+Q))\ge 2$ for some $P$ and $Q$ of $C$. If the
base point free part of $\phi ^* {\cal O}_{C'}(H)(-D_4 +P+Q)$ is
not composed with $\phi$, then the Castelnuovo-Severi inequality
implies $g \le (n-1)(nd-3) +n g(C')$, which derives $(d-1)(d-2)
\le (n-1)d$  since $g \ge 3ng(C')$ and $g(C')
=\frac{(d-1)(d-2)}{2}$. It cannot happen for $d\ge n+2$. Thus the
base point free part of $\phi ^* {\cal O}_{C'}(H)(-D_4 +P+Q)$ is
composed with $\phi$, which implies that the smooth plane curve
$C'$ of degree $d$ has a $g^1_{a\le d-2}$. It cannot occur. As a
consequence,  $\cl$ is a very ample line bundle with $h^1 (C, \cl
)=1$.

For a better understanding, we figure the correspondence of points
on curves $C$ and $C'$ in Figure 2. Here,\\

\noindent i) $\ck _C \otimes \cl ^{-1}$ : the sum of the points
being arranged as stars on the left,

\noindent ii) $H$ :  the sum of the points being arranged as black
dots and black diamonds on the right,

\noindent iii) $F_4$ : the sum of the points being arranged as
black diamonds on the right,

\noindent iv) $D_4$ : the sum of  the  points being arranged as
triangles on the left.\\

\begin{figure}

\begin{picture}(360,150)

\put(130,130){$C$}

\put(145,135){\vector(1,0){75}}

\put(175,140){$\phi$}

\put(170,123){$n:1$}

\put(230,130){$C'\subset \mathbb P^2$}

\put(95,110){$\star$}

\put(97,93){$\vdots$}

\put(95,80){$\star$}

\put(95,65){$\star$}

\put(95,50){$\star$}

\put(95,35){$\star$}

\put(95,20){$\star$}

\put(110,110){$\cdots$}

\put(110,80){$\cdots$}

\put(110,65){$\cdots$}

\put(110,50){$\cdots$}

\put(110,35){$\cdots$}

\put(110,20){$\cdots$}

\put(125,110){$\star$}

\put(127,93){$\vdots$}

\put(125,80){$\star$}

\put(125,65){$\star$}

\put(125,50){$\star$}

\put(125,35){$\star$}

\put(125,20){$\star$}

\put(140,110){$\star$}

\put(142,93){$\vdots$}

\put(140,80){$\star$}

\put(140,65){$\blt$}

\put(140,50){$\blt$}

\put(140,35){$\blt$}

\put(140,20){$\blt$}

\put(153,69){\line(-1,0){5}}

\put(153,69){\line(0,-1){47}}

\put(153,22){\line(-1,0){5}}

{\scriptsize\put(160,42){$D_4$}}

\put(80,112){\line(1,0){5}}

\put(80,112){\line(0,-1){90}}

\put(80,22){\line(1,0){5}}

 {\scriptsize\put(43,62){$\ck _C \otimes \cl ^{-1}$}}

\put(240,110){$\bullet$}

\put(242,93){$\vdots$}

\put(240,80){$\bullet$}

\put(240,65){$\bltd$}

\put(240,50){$\bltd$}

\put(240,35){$\bltd$}

\put(240,20){$\bltd$}

\put(260,69){\line(-1,0){5}}

\put(260,69){\line(0,-1){47}}

\put(260,22){\line(-1,0){5}}

{\scriptsize\put(265,42){$F_4$}}

\put(290,112){\line(-1,0){5}}

\put(290,112){\line(0,-1){90}}

\put(290,22){\line(-1,0){5}}

{\scriptsize\put(297,62){$H$}}

\put(175,65){\vector(1,0){45}}

\put(195,70){$\phi$}

\end{picture}

\caption{ Example of special line bundle\label{figure 3}}

\end{figure}
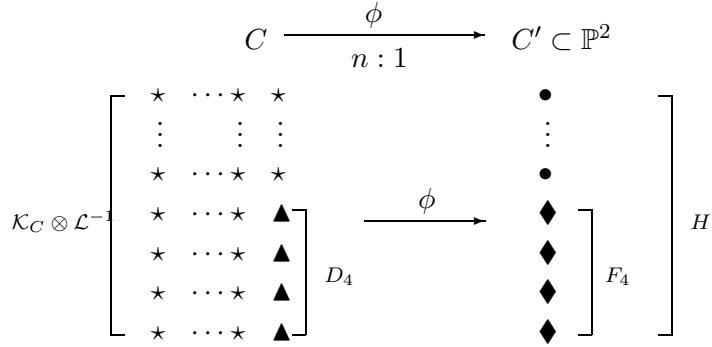

 Since $h^1 (C, \cl )=1$ and  $\ck _C \otimes \cl ^{-1}(D_4) =\phi ^* {\cal
 O}_{C'}(H)$, we have
  $h^0 (C, \cl ) -h^0 (C, \cl (-D_4 ) )\le 2$,
   and so  $\dim
\overline{\varphi_{\cl}(D_4)}\le 1$. Thus $\cl$ fails to be
normally generated due to Lemma \ref{Prop:prop 1.6}, since $h^1
(C, \cl^2 (-D))=0$.
\end{proof}
We have more various special  line bundles failing to be normally
generated on multiple coverings  due to Lemma \ref{Prop:prop 1.6}.

\begin{thm}

 Assume that $C$  admits
  a multiple covering morphism $\phi :C \rightarrow C'$
 for a smooth curve $C'$.
 And let $\cl$ be a special very ample line bundle on $C$ such
 that  $\ck _C \otimes \cl ^{-1}=\phi ^{*} \cn (-\sum ^{b}
_{i=1 }P_i)$
  for a line bundle $\cn$ on $C'$ and  $\sum ^{b} _{i=1 }P_i$
 on $C$.
 Then $\cl$ fails to be  normally generated,  if
 $h^1 (C, \cl ^2(- \sum ^{b} _{i=1 }P_i))=0$ and $ b> \frac{(b-c)(b-c+1)}{2}$,
 where $c:=h^0 (C,
\phi ^* \cn) -h^0 (C, \phi ^{*} \cn (-\sum ^{b}_{i=1} P_i))$.
\label{thm:theorem9}
   \end{thm}
 \begin{proof}
 Set $D_b :=\sum^{b}_{i=1}P_i$. Then  $h^0 (C, \cl )-h^0
(C,\cl(-D_b))=b-c$ by the Riemann-Roch theorem, which implies
$\mbox{
 dim}\overline{\varphi _\cl (D_b )} =b-c-1$.  Accordingly, the result follows
 from
Lemma \ref{Prop:prop 1.6}.
 \end{proof}

Note that, in Theorem \ref{thm:theorem9}, we get $b> 3$ by the
very ampleness of $\cl$ and the condition $ b>
\frac{(b-c)(b-c+1)}{2}$. Thus the bound $b\le 3$ in Theorem
\ref{thm:theorem6} might be optimal.

Consider the following as an application of Theorem
\ref{thm:theorem9}. Let $C'$ be a linearly normal smooth curve of
degree $d\ge 7$ in $\mathbb P ^4$. And let $C$ be a smooth curve
of genus $g$ admitting a simple covering morphism $\phi : C
\rightarrow C'$ of degree $n\ge 3$ with $g>(n-1)(nd -6)+ng(C')$.
Let $H$ be a general hyperplane section of $C'$. Set
$\cn:={\co}_{C'}(H)$. Assume $\sum^{7}_{i=1}Q_i \le H$  and $ \phi
(P_i )=Q_i$ for each $i=1,...,7$. Then $\cl:=\ck _C \otimes \phi
^* \cn ^{-1}(\sum ^{7}_{i=1} P_i)$ is a special very ample  line
bundle failing to be normally generated.
 This can be shown as follows.

 First, we claim $\cl$ is very ample with $h^1 (C, \cl ) =1$. Suppose not.
 Then  $h^0 (C,\phi ^* \cn (- \sum ^{7}_{i=1} P_i +R_1 +R_2 ))\ge 2$ for some $R_1 +R_2$ on
 $C$, since
 $h^1 (C,\cl )=h^0 (C,\phi ^* \cn ( -\sum ^{7}_{i=1} P_i))\ge 1$.
  According to the Castelnuovo-Severi inequality and the assumption
  $g>(n-1)(nd -6)+ng(C')$, the base point
 free part of $\phi ^* \cn (- \sum ^{7}_{i=1} P_i + R_1 +R_2 )$ is
 composed with $\phi$.
Let  $\{ Q_1, ..., Q_l \} := \{Q_1, ..., Q_7 \}
  -\{ \phi (R_1 ), \phi (R_2) \}$. For $n\ge  3$,
  $$h^0 (C', H (-\sum ^{l}_{i=1} Q_i ))
  \ge h^0 (C,\phi ^* \cn (- \sum ^{7}_{i=1} P_i +R_1 +R_2 ))
   \ge 2,$$ which cannot happen since $l\ge 5$ and
 the points of $\sum^{7}_{i=1}Q_i$ are  in  a general position.
 Consequently, $\cl $ is very ample with $h^1 (C, \cl ) =1$,
 whence $c:=h^0 (C, \phi ^* \cn) -h^0 (C, \phi ^{*} \cn (-\sum ^{7}_{i=1}
P_i))\ge 4$.
  According to Theorem \ref{thm:theorem9},
 $\cl$ fails to be normally generated, since the assumption
 $g>(n-1)(nd -6)+ng(C')$ implies
 $\deg \cl =2g -nd +5 \ge \frac{3g}{2} -1$  and so
 $h^1 (C, \cl ^2(- \sum ^{7} _{i=1 }P_i))=0$.

\pagestyle{myheadings} 


\markboth{}{}

\begin{thebibliography}{plain}

\bibitem[1]{Ca}
Castelnuovo, G., Sui multipli di una serie lineare di gruppi di
punti appartenente ad una curva algebrica. Rend. Circ. Mat.
Palermo {\bf 7}, (1893), 89-110.

\bibitem[2]{GL} Green, M. and Lazarsfeld, R., On the projective
 normality of complete linear series on an algebraic curve,
Invent. Math. {\bf
83} (1986), 73--90.\

\bibitem[3]{KK1} Kim, S. and Kim, Y., Projectively normal embedding of
$k$-gonal curve, Comm. in Alg. 32(1), 187-201(2004)\

\bibitem[4]{KK2} Kim, S. and Kim, Y., Normal generation of line bundles
on algebraic curves, J. of Pure and Appl. Alg. 192, 173-186(2004)\

\bibitem[5]{KKO} Kato, T., Keem, C. and Ohbuchi, A., Normal generation
of line bundles of high degrees on smooth algebraic curves, Abh.
Math. Sem. Univ. Hamburg {\bf 69} (1999), 319-333.\

\bibitem[6]{LM1} H. Lange and G. Martens,
Normal generation and presentation of line bundles of low degree
on curves. J. reine angew. Math. {\bf 356} (1985), 1-18.\

\bibitem[7]{Mat}
Mattuck, A., Symmetric prodeucts and Jacobians. Am. J. Math. {\bf
83} (1961), 189--206.

\bibitem[8]{MS} Martens, G. and Schreyer, F.O.,
Line bundles and syzygies of trigonal curves, Abh. Math. Sem.
Univ. Hamburg {\bf 56} (1986), 169--189.

\bibitem[9]{Mum} Mumford, D., Varieties defined by
quadric equations, Corso C.I.M.E. 1969, in Questions on Algebraic
Varieties, Cremonese, Rome {\bf 83} (1970), 30--100.\

\end{thebibliography}
\end{document}